\documentclass[12pt,a4paper]{amsart}
\usepackage{graphics}
\usepackage{epsfig}
\usepackage[all]{xy}
\usepackage{graphicx}
\theoremstyle{plain}
\usepackage{amssymb}
\usepackage{cite}
\usepackage[english]{babel}
\advance\hoffset-20mm \advance\textwidth40mm

\newtheorem{theorem}{Theorem}
\newtheorem{lemma}{Lemma}

\author{O.Bezushchak}
\title{Derivations of rings of infinite matrices}

\begin{document}

\maketitle

\begin{abstract} We describe derivations of several important associative and Lie rings of infinite matrices over general rings of coefficients.
\end{abstract}



\section*{Introduction}

Let $R$ be an associative ring with unit $1_R,$ and let $I$ be an infinite set.

Consider the ring  $M(I,R)$ of $(I\times I)$-matrices over the ring  $R$ having finitely many nonzero entries in each column. The ring $M(I,R)$ is isomorphic to the ring of endomorphisms of a free $R^{op}$-module of  rank  $\text{card}(I).$ Here,  $R^{op}$ is a ring that is anti-isomorphic to $R;$ and $\text{card}(I)$ is the cardinality  of the set $I.$

Consider also the subring  $M_{\infty}(I,R)<M(I,R)$ of all $(I\times I)$-matrices over  $R$ having finitely many nonzero entries, and the subring $M_{rcf}(I,R)<M(I,R)$ of all  $(I\times I)$-matrices over  $R$ having finitely many nonzero entries in each row  and in each column.

Recall that an additive mapping $d:R \to R $ is called a \textit{derivation} if $$ d(ab)= d(a) b + a d(b) \quad \text{for arbitrary elements} \quad  a, \ b \in R . $$ Let $V$ be a bi-module over a ring $R.$ An additive mapping $d:R \to V$ is called a \textit{derivation} or a $1$-\textit{cocycle} if $ d(ab)= d(a) b + a d(b) $ for arbitrary elements $ a,  b \in R .$  For an element $v\in V$ the mapping $d_v : R\to V,$ $d_v(a)=av-va$ is a derivation.

The purpose of this paper is to  determine derivations of the rings $M_{\infty}(I,R),$ $M_{rcf}(I,R),$  $M(I,R).$ Recall that all derivations of a ring form a Lie ring; see \cite{6Jacobson}.

Every derivation of the ring $R$  gives rise to a derivation of the ring $M(I,R)$ that leaves $M_{rcf}(I,R)$  and $M_{\infty}(I,R)$ invariant. Hence, the Lie ring $\text{Der}(R)$ lies in each of the Lie rings $\text{Der}(M(I,R)),$ $\text{Der}(M_{rcf}(I,R)),$ $\text{Der}(M_{\infty}(I,R)).$ 

For an element $a\in M(I,R),$ let $$\text{ad}(a):x\mapsto [a,x]=ax -xa $$ denote the \textit{inner derivation}. Since $M_{\infty}(I,R)$ is a two-sided ideal in $M_{rcf}(I,R),$ it follows that $M_{\infty}(I,R)$ is invariant under any inner derivation $\text{ad}(a),$ $a\in M_{rcf}(I,R).$

\begin{theorem}\label{Th1} \
	\begin{enumerate}
		\item[\text{(a)}]  An arbitrary derivation $d$ of the ring $M_{\infty}(I, R)$ is of the type  $$d=\emph{ad} (a) + u, \quad \text{where} \quad a\in M_{rcf}(I,R), \quad u \in \text{\emph{Der}}(R);$$
		\item[\text{(b)}]  an arbitrary derivation $d$ of the ring $M(I,R)$ $($resp. $M_{rcf}(I,R))$ is of the type $$d=\emph{ad} (a) + u, \quad \text{where} \quad a\in M(I,R) \quad  (\text{resp.} \ a\in M_{rcf}(I,R)), \quad u \in  \text{\emph{Der}}(R).$$ 
	\end{enumerate}		
\end{theorem}

In \cite{1Bezushchak}, we proved Theorem~\ref{Th1} in the case when $R$ is a field and the derivation $d$ is linear. W.~Ho{\l}ubowski and S.~\.Zurek  proved Theorem \ref{Th1} under the assumptions that $(i)$ the set $I$ is countable, $(ii)$ the ring $R$ is commutative, and $(iii)$ the derivation $d$ is $R$-linear; see \cite{5Hol_Zurek}.

An arbitrary associative ring $R$ gives rise to the Lie ring $$R^{(-)}=(R, \ [a,b]=ab-ba).$$

Using the proof of Herstein's Conjectures by K.I.~Beidar, M.~Bre\v{s}ar, M.A.~Chebotar and  W.S.~Martindale (see \cite{2Bei_Bre_Cheb_Mart_1, 3Bei_Bre_Cheb_Mart_2, 4Bei_Bre_Cheb_Mart_3}) and Theorem~\ref{Th1},  we obtain descriptions of derivations of Lie rings  $$\mathfrak{gl}_{rcf}(I,R)=M_{rcf}(I,R)^{(-)}, \quad \mathfrak{gl}(I,R) =M(I,R)^{(-)}, \quad \mathfrak{sl}_{\infty}(I,R)=[\mathfrak{gl}_{\infty}(I, R),\mathfrak{gl}_{\infty}(I, R)]$$ under the assumption that $\frac{1}{2}\in R.$

\begin{theorem}\label{Th2} \
	\begin{enumerate}
		\item[\text{(a)}]  An arbitrary derivation $d$ of the Lie ring $\mathfrak{sl}_{\infty}(I,R)$ is of the type $$d=\emph{ad} (a) + u, \quad \text{where} \quad a\in \mathfrak{gl}_{rcf}(I,R), \quad u \in \text{\emph{Der}}(R);$$
		
	\item[\text{(b)}]  an arbitrary derivation $d$ of the Lie ring $\mathfrak{gl}(I,R)$ $($resp. $\mathfrak{gl}_{rcf}(I,R))$ is of the type $$d=\emph{ad} (a) + u, \quad \text{where} \quad a\in \mathfrak{gl}(I,R) \quad (\text{resp.} \ a\in  \mathfrak{gl}_{rcf}(I,R)), \quad u \in  \text{\emph{Der}}(R).$$ 
	\end{enumerate}		
\end{theorem}

For an element $a\in R$ and indices $ i,j\in I$ let $e_{ij}(a)$ denote the  $(I\times I)$-matrices  having the element $a$ at the intersection of the $i$-th row and the $j$-th column, all other entries are equal to zero.

The following lemma is straightforward.

\begin{lemma}\label{lemma1} Let $A=(a_{ij})_{i,j\in I} \in M(I,R),$ $a_{ij}\in R.$ Then $$ A=e_{kk}(1_R)A + Ae_{kk}(1_R)$$ if and only if $ a_{ij}=0 $ whenever $ i \neq k, $ $ j \neq k$ or $ i=j=k.$ \end{lemma}

Let $\mathbb{Z}$  be the ring of integers. Then $\mathbb{Z} \cdot 1_R$ is a subring of the ring $R.$ The ring $M(I,R)$ is a bimodule over the ring $M_{\infty} (I, \mathbb{Z}\cdot 1_R).$

\begin{lemma}\label{lemma2} For an arbitrary derivation $d:M_{\infty} (I, \mathbb{Z}\cdot 1_R) \to M (I, R)$ there exists a matrix $v\in M_{\infty}(I,R)$ such that $$ \big(d-d_v\big) \big(e_{kk}(1_R)\big)=0 $$ for an arbitrary index $k\in I.$ \end{lemma}

\begin{proof} Let $$ d\big(e_{kk}(1_R)\big) =(a_{ij}^{(k)})_{i,j\in I}, \quad k\in I, \quad a_{ij}^{(k)} \in R. $$ We have $$ d\big(e_{kk}(1_R)\big)= d\big(e_{kk}(1_R)\big) e_{kk}(1_R) + e_{kk}(1_R)d\big(e_{kk}(1_R)\big) .$$ By Lemma \ref{lemma1}, $a_{ij}^{(k)}=0$ whenever $ i \neq k, $ $ j \neq k$ or $ i=j=k.$  Choose distinct indices $p,q \in I.$ Then $e_{pp}(1_R)e_{qq}(1_R) =0$ implies $$ (a_{ij}^{(p)})_{i,j\in I} \ e_{qq}(1_R)+ e_{pp}(1_R)\  (a_{kl}^{(q)})_{k,l\in I}=0. $$ The $(p,q)$-entry of the matrix on the left hand side is $$ a_{pq}^{(p)}+ a_{pq}^{(q)}=0. $$ Let $X=(x_{ij})_{i,j\in I},$ $x_{ij} \in R.$ The $(p,q)$-entry of the matrix $[X, e_{kk}(1_R)]$ is equal to $0$ if $p\neq k , $ $q\neq k;$ is equal to $-x_{pq}$ if $k=p;$ and  to $x_{pq}$ if $k=q.$ The diagonal entries of the matrix $ [X, e_{kk}(1_R)] $ are equal to zero. 

Define $a_{ij} =a_{ij}^{(j)}$ for $i \neq j;$ $a_{ii}=0;$ and $v=(a_{ij})_{i,j\in I}.$ By the above, $$ d\big(e_{kk}(1_R)\big) =(a_{ij}^{(k)})_{i,j\in I}= \big[v, e_{kk}(1_R)\big] $$ for an arbitrary index $k\in I.$

The $j$-th column of the matrix $v$ is equal to the $j$-th column of the matrix $ d\big(e_{kk}(1_R)\big) \in M(I,R).$ Hence, only finitely many entries in the $j$-th column are different from zero, hence $v\in M(I,R).$ This completes the proof of the lemma.
\end{proof}
                              
\begin{lemma}\label{lemma3} Let $v\in M(I,R).$ Then 
\begin{enumerate}
	\item[\text{(a)}]  $\big[ v, M_{\infty}(I, R)\big]\subseteq M_{rcf}(I,R)$ if and only if $v\in M_{rcf}(I,R);$
	\item[\text{(b)}]  if   $v\in M_{rcf}(I,R),$ then in fact $\big[ v, M_{\infty}(I, R) \big] \subseteq M_{\infty}(I,R).$
\end{enumerate}		 \end{lemma}

\begin{proof} The assertion (b) immediately follows from the fact that $M_{\infty}(I, R)$ is a two-sided ideal in the ring $M_{rcf}(I, R).$ 
	
Let $$ v=(a_{ij})_{i,j\in I}\in M(I,R), \quad \big[ v,e_{kk}(1_R) \big] \in  M_{rcf}(I,R) \quad \text{for any} \quad k\in I.$$ All entries in the $k$-th row of the matrix $v,$ except for the diagonal one, are negatives of the corresponding entries in the $k$-th row of the matrix $[v,e_{kk}(1_R)].$ Since $\big[ v, e_{kk}(1_R) \big]  \in M_{rcf}(I,R),$ it follows that only finitely many entries of the $k$-th row of the matrix $v$ are different from zero. Hence,   $v\in M_{rcf}(I,R).$ This completes the proof of the lemma.
\end{proof}

Now, we are ready to prove Theorem \ref{Th1}. 
\begin{proof} Let $d$ be a derivation of the ring $M_{\infty}(I,R).$ By Lemma \ref{lemma2}, there exists a matrix $v\in M(I,R)$ such that $d\,' =d-d_v$ maps all matrix units $e_{kk}(1_R)$ to zero. This implies that $$ d\,'\big(e_{ij}(R)\big)\subseteq e_{ij}(R) $$ for all indices $i,  j \in I.$
	
Define  the additive mapping $d_{ij}:R\to R$ via $$ d\,' \big( e_{ij}(a) \big) =e_{ij}\big( d_{ij}(a)\big) .$$ Clearly, $d_{ij}(1_R)=0.$ Let $a,b\in R;$ $i,  j,  k \in I.$ We have $ e_{ij}(ab)=e_{ik}(a)e_{kj}(b).$ Hence, $$ d_{ij}(ab)=d_{ik}(a) \cdot b + a \cdot d_{kj}(b). $$ Let $b=1_R.$ Then $d_{ij}(a)=d_{ik}(a).$ This implies that $d_{ij}$ does not depend on $i,$ $j.$

Define $d=d_{ij};$ $i,  j \in I.$ The mapping $d$ is a derivation of the ring $R.$ Hence, $d\,' \in \text{Der}(R).$ All derivations from $\text{Der}(R)$ map $M_{\infty}(I,R)$ into itself. Therefore, $$ d_v\big( M_{\infty}(I,R) \big) \subseteq M_{\infty}(I,R).$$ By Lemma \ref{lemma3}(a), it implies that $v\in M_{rcf}(I,R)$ and completes the proof of the first part of the theorem.

Let $d$ be a derivation of the ring $M(I,R).$ Arguing as above, we find a matrix $v\in M(I,R)$ such that $u=d- d_v \in \text{Der}(R).$ If $d$ is a derivation of the ring $ M_{rcf}(I,R),$ then the mapping $d_v$ maps $M_{\infty}(I,R)$ to $M_{rcf}(I,R).$ By Lemma \ref{lemma3}(a), $v\in M_{rcf}(I,R).$ This completes the proof of the theorem.
\end{proof}

\end{document}